\documentclass[12pt]{amsart}
\usepackage[shortlabels]{enumitem}

\usepackage{graphicx,color,bbm}
\usepackage{amsfonts}
\usepackage{amssymb}
\usepackage{amsthm}
\usepackage{amsmath}
\usepackage{amscd}
\usepackage{mathrsfs}
\usepackage[pdftex, colorlinks]{hyperref}
\hypersetup{colorlinks=true,  
    linkcolor=black,          
    citecolor=black,        
    filecolor=black,      
    urlcolor=black           
    }

\newtheorem{thm}{Theorem}[section]
\newtheorem{lem}[thm]{Lemma}

\newtheorem{cor}[thm]{Corollary}

\newtheorem*{thm*}{Theorem}

\newtheorem{rmk}[thm]{Remark}

\newcommand{\ce}{2400}
\newcommand{\cg}{2500}
\newcommand{\cf}{2500}
\newcommand{\LL}{\mathcal L}
\newcommand{\CC}{\mathcal C}
\newcommand{\logeps}{|\mathrm{log}\,\epsilon|}
\newcommand{\step}[1]{\subsection{#1}}
\DeclareMathOperator{\var}{var}
\DeclareMathOperator{\sign}{sgn}

\newcommand{\BV}{\text{BV}}

\newcommand{\om}{\omega}

\newcommand{\Om}{\Omega}

\newcommand{\sig}{\sigma}
\newcommand{\lam}{\lambda}
\newcommand{\ka}{\kappa}

\newcommand{\ep}{\epsilon}

\newcommand{\PP}{\ensuremath{\mathbb P}}

\newcommand{\hpt}{\ensuremath{\mathscr{H}_p^t}}

\begin{document}
\title[Random metastable maps]{
Lyapunov exponents for transfer operator cocycles of metastable maps: a quarantine approach
}
\author{Cecilia Gonz\'alez-Tokman and Anthony Quas}
\address{C. Gonz\'alez-Tokman: School of Mathematics and Physics,
The University of Queensland,
St Lucia, QLD 4072 Australia; \texttt{cecilia.gt@uq.edu.au}
\and
A. Quas: Department of Mathematics and Statistics,
University of Victoria,
Victoria, BC,
CANADA V8W 2Y2; \texttt{aquas@uvic.ca}}
\begin{abstract}
This works investigates the Lyapunov--Oseledets spectrum of transfer operator cocycles associated to  one-dimensional random
\emph{paired tent maps}
depending on a parameter $\epsilon$, quantifying the strength of the \emph{leakage} between two nearly invariant regions. We show that the system exhibits metastability, and identify the second Lyapunov exponent $\lambda_2^\ep$ within an error of order $\ep^2|\log \ep|$. 
This approximation agrees with the naive prediction provided by a time-dependent two-state Markov chain. 
Furthermore, it is shown that $\lambda_1^\ep=0$ and $\lambda_2^\ep$ are simple, and the only exceptional
Lyapunov exponents of magnitude greater than $-\log2+ O(\log\log\tfrac 1\epsilon/\log\tfrac 1\epsilon)$.
\end{abstract}

\maketitle


\section{Introduction}
A dynamical system is called metastable if there is a time scale over which its statistical behaviour appears to be in equilibrium, but such an equilibrium differs from the long term (asymptotic) equilibrium for the system. Metastability has been studied in connection with phase transitions in physics, the existence of isomers in chemistry and the presence of slowly mixing regions in the ocean. 
In mathematics, metastability has been initially studied from a probabilistic perspective; see for instance the monographs \cite{FreidlinWentzel,BovierdenHollander} and references therein. 

In  \cite{KellerLiverani09}, Keller and Liverani pioneered the study of metastability from a (chaotic) dynamical systems perspective. In this context, a natural way to define metastability is through spectral properties of transfer operators. When the system is ergodic, the transfer operator has a unique eigenvector of eigenvalue one, describing the system's long term behaviour. Metastability is associated with the presence of an eigenvalue of magnitude near, but strictly less than one. Its corresponding eigenvector, sometimes called strange eigenmode, characterises the possible behaviours of the system in relatively long, yet not asymptotic, time scales.
This point of view has been developed by various authors  in \cite{FroylandStancevic10,GTHuntWright,DolgopyatWright,Keller12,BahsounVaienti,FergusonPollicott}.

While the spectrum captures information about iterates of a \emph{single operator},
if the forcing is stationary, the spectral picture can often be replaced by the so-called Oseledets decomposition, discovered by Oseledets in his breakthrough paper \cite{Oseledets}, and adapted to the context of transfer operators by Froyland, Lloyd and Quas in \cite{FLQ1,FLQ2}. In this setup, the largest Lyapunov exponent is zero, corresponding to a (random) invariant measure for the system, and a natural way to define metastability is through the existence of a negative Lyapunov exponent $\lambda_2$ near zero.
Then, the associated Oseledets spaces provide a \emph{coherent structure} which decays at the slow exponential rate $\lambda_2$.  

While the connection between autonomous and non-autonomous concepts of metastability is transparent at a conceptual level, the identification of examples of random metastability has proved more elusive.
In \cite{FroylandStancevic13}, the authors investigate connections of this phenomenon with escape rates from random sets.
  Very recently, there has been some further progress: In \cite{Horan1} and \cite{Horan2}, Horan provides bounds on the second Lyapunov exponent for a class of random interval maps using cone techniques, and in \cite{Crimmins}, Crimmins shows a result on stability for hyperbolic Oseledets splittings for quasi-compact cocycles, modelled on the famous stability result of Keller and Liverani \cite{KellerLiverani99}.

A significant obstacle in the construction of examples of random metastable systems is that Lyapunov exponents are asymptotic quantities, which are generally difficult to rigorously bound, especially from below. In fact, it has been recently  shown by the authors in \cite{GTQ-jems20} that small perturbations in the dynamics can result in drastic changes in the Lyapunov--Oseledets spectrum of a system of transfer operators.

In this work, we consider random compositions of a class of one-dimensional piecewise smooth expanding maps, and investigate the corresponding transfer operator cocycle in the space of functions of bounded variation.
This framework was introduced by Buzzi in \cite{Buzzi00}, where he showed existence of random absolutely continuous invariant measures, corresponding to the first Oseledets space with associated Lyapunov exponent equal to zero.

 Here we focus on the so-called \emph{paired tent maps}, defined in \eqref{eq:pairedtentmap} following \cite{Horan1}, and introduce a parameter $\epsilon$, to quantify the strength of the \emph{leakage} between nearly invariant components.
 We show that for small $\ep>0$ the second Lyapunov exponent $\lambda_2^\ep$ is of order $\ep$. In fact, we identify $\lambda_2^\ep$ within an error of order $\ep^2|\log \ep|$. 
This approximation agrees with the naive prediction that a simple coarse graining by a time-dependent two-state Markov chain would yield, corresponding to a two-bin Ulam approximation scheme.
In addition, we show the corresponding Oseledets space is 
 one-dimensional, and in some sense close to a step function. Therefore, it separates the two nearly invariant components of the system, in agreement with the (random version of the) Dellnitz--Froyland Ansatz. 
Furthermore, it is shown that, apart from 0 and $\lambda_2^\ep$, there are no other Lyapunov exponents of magnitude greater than $-\log2+ O(\log\log\tfrac 1\epsilon/\log\tfrac 1\epsilon)$. It follows from \cite{FLQ2} that all exceptional Lyapunov exponents are greater than $-\log2+ O(\ep)$.

The main idea of the proof is to keep track of the \emph{mass} that has switched from one nearly invariant component to the other. 
One could think of this bookkeeping procedure as a \emph{quarantine} period, during which the \emph{leaked} densities are 
kept separate from the rest, until they have been mixed. In this way, it is possible to identify an invariant subset of densities, $\mathcal C$, in some sense close to the invariant densities for the $\ep=0$ map, which is shown to contain a slowly decaying coherent structure.  

While our main results are stated in the context of random dynamical systems, that is, assuming there is an ergodic process $\sigma$ driving the dynamics, many of our arguments carry over to the non-stationary setting. Indeed, one may define orbit-wise Lyapunov exponents as in \eqref{eq:lexp}. Then, 
the set $\mathcal C$ of \S\ref{s:invariantSet} remains invariant in the non-stationary situation, and
 the results of \S\ref{s:1st2ndexp} --except for the use of Birkhoff ergodic theorem-- and \S\ref{s:3rdLexp} would yield estimates  on first, second and third Lyapunov exponents of the non-stationary cocycle.

\section{Notation and main results}
For $0\le a,b\le 1$, we define a \emph{paired tent map} $T_{a,b}$ to be the piecewise linear map on $J= [-1,1]$ sending $-1$ to $-1$, $-\tfrac 12$ to $b$ and $0^-$ to $-1$; $0^+$ to 1; $\tfrac 12$ to $-a$ 
and 1 to 1, as introduced in \cite{Horan1}. That is,
\begin{equation}\label{eq:pairedtentmap}
T_{a,b}(x) = \begin{cases}
2(1+b)(x+1) - 1, & x\in [-1,-1/2], \\
-2(1+b)x - 1, & x\in [-1/2,0), \\
0, & x = 0, \\
-2(1+a)x + 1, & x\in (0,1/2], \\
2(1+a)(x-1) + 1, & x\in [1/2,1].
\end{cases} 
\end{equation}

In the case $a=b=0$, the system is reducible, consisting of
of a pair of tent maps on disjoint intervals, $[-1,0]$ and $[0,1]$. 
For small positive $a$ and $b$, there is a small amount of \emph{leakage} between the two halves:
points near $-\tfrac 12$ leak
to the right half, while points near $\tfrac 12$ leak to the left half. 
We can think of $a$ and $b$ as leakage controls.

We study cocycles of paired tent maps, driven by an ergodic, invertible, probability preserving transformation $\sig:\Om \to \Om$, with the leakage controls $a(\om)$ and $b(\om)$ scaled by a parameter $\epsilon$,
and look at the asymptotics of the Lyapunov exponents of the associated Perron-Frobenius
operator cocycle as $\epsilon$ approaches 0. 

We let $\text{BV}$ denote the Banach space of functions of bounded variation on 
$[-1,1]$, where, as usual, we identify functions that disagree on a set of measure 0. 
For convenience later, we equip \text{BV} with a non-standard norm (although it is equivalent 
to the standard one), namely we define
$$
\|f\|_\text{BV}=\max(\var_{0^c}f,\|f\|_1),
$$
where
$\var_{0^c}f$ denotes $\var_{[-1,0)}f+\var_{(0,1]}f$. 
For any $I\subset [-1,1]$, we let $\int_I f:=\int_I f dm$, where $m$ is the Lebesgue measure, and $\|f\|_1=\int_{[-1,1]}|f|dm$.
We also let $\sign=\mathbf{1}_{[1,0]}-\mathbf{1}_{[-1,0]}$ be the sign function.

The main result of this work is the following.
\begin{thm}\label{mainthm}
Let $\sigma$ be an invertible ergodic measure-preserving transformation of a probability space 
$(\Omega,\PP)$. Let $a$ and $b$ be non-zero measurable functions from $\Omega$ to $[0,1]$,
and for each $\omega\in\Omega$ and $\epsilon\in(0,1)$, let $T^\epsilon_\omega$ denote the map $T_{\epsilon a(\omega),\epsilon b(\omega)}$.
Let $\LL^\epsilon_\omega$ denote the corresponding Perron-Frobenius operator. 

Then the leading two Lyapunov exponents of the cocycle ${\mathcal L^\epsilon}^{(n)}_\omega$ 
(acting on BV)
are $\lambda^\ep_1=0$ and $\lambda_2^\ep=-\epsilon\int(a(\omega)+b(\omega))\,d\PP(\omega)+O(\epsilon^2\logeps)$.
Furthermore, $\lambda^\ep_1$ and $\lambda^\ep_2$ have multiplicity one and the remainder of the Lyapunov spectrum is bounded above by $-\log 2+O(\log\log\tfrac 1\epsilon/\log\tfrac 1\epsilon)$. 
\end{thm}

\begin{rmk}\label{rmk:met}
Horan, in his thesis (See \cite[Remark 4.5.7]{HoranThesis}), discusses the fact that the known versions
of operator-valued METs when applied to Perron-Frobenius operators on BV only prove the existence of
(measurable) Oseledets spaces in cases where the map $\omega\mapsto T_\omega$ has countable range. He
explains that this is not an artefact of the proof, but rather a consequence of the geometry of BV. Indeed, he
points out that $\{\mathcal L_{T_{a,b}}\mathbf 1\}$ is a uniformly discrete set as $(a,b)$ run over $[0,1]^2$. This means that
even the map $(a,b)\mapsto \LL_{T_{a,b}}\mathbf 1$ is not measurable with respect to the Borel $\sigma$-algebra
on BV. 

In the absence of an MET for more general choices of $\omega\mapsto T_\omega$, a natural way to interpret the statement of Theorem~\ref{mainthm} is by defining the $k$-th Lyapunov exponent $\lam_k^\ep(\om)$ as the supremum over $k$-dimensional subspaces $V$ of the minimal growth rate of $\|\LL_\omega^{(n)}f\|_\text{BV}$, for $f\in V\setminus \{0\}$. That is,
\begin{equation}\label{eq:lexp}
\lam_k^\ep(\om)=\sup_{V\in \mathcal G_k(\BV)} \inf_{f\in V\setminus \{0\}} \limsup_{n\to\infty}
\tfrac1n\log \|\LL_\omega^{\ep(n)}f\|_\text{BV},
\end{equation}
where $\mathcal G_k(\BV)$ is the Grassmannian of $k$-dimensional subspaces of BV.
In this context, our arguments show that for $\PP$-a.e. $\om\in\Om$, $\lam_1^\ep(\om)=0$, $\lambda_2^\ep(\om)=-\epsilon\int(a(\omega)+b(\omega))\,d\PP(\omega)+O(\epsilon^2\logeps)$ and
$\lambda_3^\ep(\om)\leq -\log 2+O(\log\log\tfrac 1\epsilon/\log\tfrac 1\epsilon)$.

In Section~\ref{s:METs}, we give an alternative interpretation of Theorem~\ref{mainthm} in the context of multiplicative ergodic theory. In particular, it implies measurability of the top Oseledets spaces when regarded as functions from $\Om$ to $L^1$.
\end{rmk}

\subsection{Connection with time-dependent two-state Markov chains}
Let $A^\epsilon_\omega$ denote the matrix $\begin{pmatrix}1-\epsilon b(\omega)&\epsilon a(\omega)\\
\epsilon b(\omega)&1-\epsilon a(\omega)\end{pmatrix}$, with $a$ and $b$ as in the statement of Theorem~\ref{mainthm}. Let ${A^\epsilon}^{(n)}_\omega=A^\epsilon_{\sig^{n-1}\omega} \circ \dots \circ A^\epsilon_\omega$ be the 
corresponding matrix cocycle. This is the sequence of transition (Ulam) matrices one would obtain by coarsely partitioning $J=[-1,1]$ into $J^-=[-1,0]$
and $J^+=[0,1]$.

 The two Lyapunov exponents 
of the cocycle ${A^{\epsilon}_\omega}^{(n)}$ are 0 and 
$\int\log(1-\epsilon(a(\omega)+b(\omega))\,d\PP(\omega)$. To see this, notice that $\begin{pmatrix}
1&1\end{pmatrix}A_\omega^\epsilon=\begin{pmatrix}1&1\end{pmatrix}$, ensuring that 0 is a Lyapunov
exponent of the dual cocycle, and hence of the primal cocycle. Also, the sum of the Lyapunov exponents is
given by the average value of the log-determinant, so that the second Lyapunov exponent is
$\int\log(1-\epsilon(a(\omega)+b(\omega)))\,d\PP(\omega)$ as claimed. By taking a Taylor expansion, 
we see that the second Lyapunov exponent is $-\epsilon\int (a(\omega)+b(\omega))\,d\PP(\omega)+O(\epsilon^2)$
so that the top two Lyapunov exponents of the paired tent map cocycle $\LL^\epsilon_\omega$ of Theorem~\ref{mainthm} agree with those of its coarse Ulam approximation up to order $\epsilon^2\logeps$.

\section{Proof of the main theorem}
In this section, we present the proof of Theorem~\ref{mainthm}. 
Let $\epsilon>0$ be fixed and sufficiently small (with the precise conditions specified below, but where 
$\epsilon<\tfrac{1}{2000}$ suffices).
Let $k$ be such that $2^k\epsilon<\tfrac 14$ but 
$2^{k+1}\epsilon\ge \tfrac 14$. 

For a fixed $\omega$ and $\epsilon$, let $H^+_\omega=
[\tfrac 1{2(1+\epsilon a(\omega))},1-\tfrac 1{2(1+\epsilon a(\omega))}]$, the set of points that leak from
$J^+:=[0,1]$ to $J^-:=[-1,0]$ under $T^\epsilon_\omega$ and similarly let 
$H^-_\omega=[-1+\tfrac 1{2(1+\epsilon b(\omega))},-\tfrac 1{2(1+\epsilon b(\omega))}]$, the set of points 
leaking from  $J^-$
to $J^+$.
We set $H_\omega=H_\omega^+\cup H_\omega^-$.

We consider the space of $(k+1)$-tuples of $\text{BV}$ functions: $X=(\text{BV}[-1,1])^{k+1}$ and define for each
$\omega$, an operator $\Lambda_\omega$ on $X$ by
\begin{align*}
\Lambda_\omega(f_0,\ldots,f_k)=(\mathcal L_\omega (f_0\mathbf 1_{H_\omega^c}+f_k),
\mathcal L_\om(\mathbf 1_{H_\omega}f_0),\mathcal L_\om f_1,\ldots
,\mathcal L_\om f_{k-1}),
\end{align*}
where $\mathcal L_\omega=\mathcal L^\epsilon_\omega$. 
Note that $H_\omega$ and $\Lambda_\omega$ also depend on $\ep$, but since
 $\ep$ will be fixed throughout the proofs, we do not write this dependence explicitly. 

The above should be interpreted as follows: any two visits to the hole are separated by at least $k$ steps. 
The overall density is represented by $f_0+\ldots+f_k$. For $1\le j\le k$, the term $f_j$ 
represents the part of the density coming from mass that leaked through one of the holes $j$ steps ago,
while $f_0$ is the part of the density coming from mass that has not recently passed through a hole,
so that, for example, the 0-term of $\Lambda_\omega(f_0,\ldots,f_k)$ consists of two parts: the image of
the part of $f_0$ that did not pass through the hole, together with the image of $f_k$, the contribution 
to the density that passed through the hole $k$ steps previously.

\step{Construction of an invariant set $\CC$}\label{s:invariantSet}
Let $\CC$ be the collection of elements $\bar f=(f_0,\ldots,f_k)$ of $X$ satisfying the following conditions:
\begin{enumerate}[($\CC$1)]
\item $\var_{0^c} f_j\le 4\cdot \big(2(1-39\epsilon)\big)^{-j}\|f_0\|_1$ for $j=1,\ldots,k$;\label{it:varfj}
\item $\|f_j\|_1\le 3(1-39\epsilon)^{-j}\epsilon\|f_0\|_1$ for $j=1,\ldots,k$;\label{it:L1fj}
\item $\var_{0^c}f_0\le 33\epsilon \|f_0\|_1$.\label{it:varf0}
\end{enumerate}

We now show $\CC$ is invariant under each $\Lambda_\omega$.
Let $\bar f=(f_0,\ldots,f_k)\in\mathcal C$, $\omega\in\Omega$
and  $\bar g=(g_0,\ldots,g_k)=\Lambda_\omega\bar f$.
We first prove that $\|g_0\|_1\ge (1-39\epsilon)\|f_0\|_1$. 

Since $\var_{0^c}f_0\le 33\epsilon \|f_0\|_1$, we see that 
for any $x\in J^+$, 
\begin{equation}
\left|f_0(x)-\textstyle\int_{J^+} f_0\right|\le33\epsilon \|f_0\|_1,
\label{eq:f0J+}
\end{equation}
and similarly for any $x\in J^-$, 
\begin{equation}\left|f_0(x)-\textstyle\int_{J^-} f_0\right|\le33\epsilon \|f_0\|_1.
\label{eq:f0J-}
\end{equation}

If $\int_{J^+}|f_0|\ge 33\epsilon\|f_0\|_1$, then $|f_0|$ takes a value at least 
$33\epsilon\|f_0\|_1$ in $J^+$, so that \ref{it:varf0} implies $f_0$ is non-negative throughout $J^+$
or $f_0$ is non-positive throughout $J^+$. In either case we see that 
\begin{align*}
\int_{J^+}|\LL_\omega
(f_0\mathbf 1_{H_\omega^c})|&=\left|\int_{J^+}\LL_\omega(f_0\mathbf 1_{H_\omega^c})\right|\\
&=\left|\int_{J^+}f_0\mathbf 1_{H_\omega^c}\right|=\int_{J^+}|f_0|-\int_{H_\omega^+}|f_0|\\
&\ge \int_{J^+}|f_0|-m(H_\omega^+)\left(\int_{J^+}|f_0|+
33\epsilon\|f_0\|_1\right)\\
&\ge \int_{J^+}|f_0|-2\epsilon\|f_0\|_1,
\end{align*}
where we used the fact that $m(H_\omega^+)\le \epsilon$ and we assumed 
that $\epsilon <\tfrac1{33}$. Of course if $\int_{J^+}|f_0|<33\epsilon\|f_0\|_1$, then
$\int_{J^+}|\LL_\omega(1_{H_\omega^c}f_0)|\ge 0> \int_{J^+}|f_0|-33\epsilon\|f_0\|_1$.

A completely symmetric argument holds for $J^-$. 
Assume $\epsilon<\tfrac 1{66}$. Then, at least one of 
$\int_{J^+}|f_0|\ge 33\epsilon\|f_0\|_1$ and $\int_{J^-}|f_0|\ge 
33\epsilon\|f_0\|_1$ is satisfied. Hence, summing
the relevant inequalities, we obtain
\begin{equation*}
\left\|\LL_\omega f_0\mathbf 1_{H_\omega^c}\right\|_1\ge \|f_0\|_1(1-35\epsilon).
\end{equation*}

We then have 
\begin{equation*}
\begin{split}
\|g_0\|_1&=
\|\LL_\omega (f_0\mathbf 1_{H_\omega^c}+f_k)\|_1\\
&\ge \|\LL_\omega (f_0\mathbf 1_{H_\omega^c})\|_1-\|\LL_\omega f_k\|_1\\
&\ge(1-35\epsilon)\|f_0\|_1-\|f_k\|_1\\
&\ge \Big(1-\left(35+3/(1-39\epsilon)^k\right)\epsilon\Big)\|f_0\|_1.
\end{split}
\end{equation*}

At this point, recalling that $k\geq \log_2(\tfrac 1\epsilon)-3$, 
let us assume that $\epsilon$ is sufficiently small that $3/(1-39\epsilon)^{k}<4$,
so that 
\begin{equation}
\|g_0\|_1\ge (1-39\epsilon)\|f_0\|_1.
\end{equation}

Now if for $1\le j<k$, $\var_{0^c}f_j\le 4/\big(2(1-39\epsilon)\big)^j\|f_0\|_1$,
then 
\begin{align*}
\var_{0^c}g_{j+1}&=\var_{0^c}\LL_\omega f_j\\
&\le \tfrac 12\var_{0^c}f_j\\
&\le 4/\big(2(1-39\epsilon)\big)^{j+1}
(1-39\epsilon)\|f_0\|_1\\
&\le 4/\big(2(1-39\epsilon)\big)^{j+1}\|g_0\|_1,
\end{align*}
so that $\bar g$ satisfies \ref{it:varfj} for $j=2,\ldots,k$. 

Similarly, since $\bar f$ satisfies \ref{it:L1fj}, then for 
$1\le j<k$
\begin{align*}
\|g_{j+1}\|_1&=\|\LL_\omega f_j\|_1\\
&\le \|f_j\|_1\le 3(1-39\epsilon)^{-j} \ep\|f_0\|_1\\
&\le 3(1-39\epsilon)^{-(j+1)} \ep \|g_0\|_1,
\end{align*}
so that $\bar g$ satisfies \ref{it:L1fj} for $j=2,\ldots,k$. 

To establish \ref{it:L1fj} for $g_1$, recall
$g_1=\LL_\omega(f_0\mathbf 1_{H_\omega})$ so that
$\|g_1\|_1\le\|f_0\mathbf 1_{H_\omega}\|_1$. 
From \eqref{eq:f0J+}, we see $|f_0|$ takes values at most $|\int_{J^+}f_0|+33\epsilon\|f_0\|_1$ on $H^+$
and similarly takes values at most $|\int_{J^-}f_0|+33\epsilon\|f_0\|_1$ on $H^-$.
Hence $\|g_1\|_1\le m(H_\omega)\big(\|f_0\|_1+33\epsilon \|f_0\|_1\big)\le
3\epsilon\|g_0\|_1$ (where we assumed that $\epsilon$ was sufficiently small that
$2\epsilon(1+33\epsilon)/(1-39\epsilon)<3\epsilon$), thereby establishing \ref{it:L1fj} for $j=1$.

We now show \ref{it:varfj} for $j=1$. We have $\var_{0^c}g_1=\var_{0^c}(\LL_\omega
\mathbf 1_{H_\omega}f_0)
\le \tfrac 12\var_{H_\omega}(f_0)+|f_0(\tfrac 12)|+|f_0(-\tfrac 12)|
\le \tfrac {33}2\epsilon\|f_0\|_1+(|\int_{J^+}f_0|+33\epsilon\|f_0\|_1)
+(|\int_{J^-}f_0|+33\epsilon\|f_0\|_1)\le (1+\tfrac{165}2\epsilon)\|f_0\|_1\le
2/(1-39\epsilon)\|g_0\|_1$ as required. 

Finally, we establish \ref{it:varf0} for $\bar g$.
We have $\var_{0^c}g_0=\var_{0^c}\LL(\mathbf 1_{H_\omega^c}f_0+f_k)
\le \tfrac 12(\var_{0^c}f_0+\var_{0^c}f_k)\le \tfrac{33}2\epsilon\|f_0\|_1+
2\cdot 2^{-k}(1-39\epsilon)^{-k}\|f_0\|_1$. By the choice of $k$, we have 
$2^{-k}\le 8\epsilon$, so that provided $\tfrac{33}2+16(1-39\epsilon)^{-k}<33(1-39\epsilon)$, we have
$\var_{0^c}g_0\le 33(1-39\epsilon)\ep\|f_0\|_1\le 33\ep\|g_0\|_1$. 

This completes the verification of the invariance of $\CC$ under $\Lambda_\omega$.

\step{Estimation of first and second Lyapunov exponents}\label{s:1st2ndexp}
Up to this point, we have written explicit bounds in terms of $\epsilon$, since the bounds are inter-dependent
and this allows us to verify that they are simultaneously satisfiable. At this point, we switch to using the more compact
$O(\cdot)$ notation (that is $A=O(B)$ means that there is a universal constant $C$
such that for all $\bar f\in \CC$, the inequality $|A(f)|\le C|B(f)|$ is satisfied).
The rationale for using this notation from here on is that from now on,
each inequality will be seen to follow from earlier estimates. Additionally, we can conveniently write
expressions like $\phi(\Lambda\bar f)=\phi(f)(1+c\epsilon+O(\epsilon^2\logeps))$.
It would not be hard to write explicit constants in place of the $O(\cdot)$ notation if desired.

For $\bar f=(f_0,\ldots,f_k)\in X$, we let 
$$\Phi(\bar f )=f_0+\ldots+f_k\in \BV.$$

\subsubsection{Estimation of the first Lyapunov exponent}\label{s:1stExp}
If $f\in \BV$,
then for a suitably large constant $c$,
$f$ may be expressed as $f=g+h$ where $g=f-c$, $h=c$, and $\var_{0^c}g<33\ep \|g\|_1$.
Write $\bar g=(g,0,\ldots,0)$ and $\bar h=(h,0,\ldots,0)$. Then, $\bar g,\bar h\in\CC$. 
It is straightforward to verify that for any $\bar f=(f_0, \dots, f_k)\in X$, $\sum_{i=0}^k \|(\Lambda_\omega \bar f)_i\|_1\le \sum_{i=0}^k \|f_i\|_1$.
Hence,  $\|(\Lambda_\omega^{(n)}\bar g)_0\|_1$
and $\|(\Lambda_\omega^{(n)}\bar h)_0\|_1$ are uniformly bounded above in $n$, so that
by the definition of $\CC$, $\|\Phi(\Lambda_\omega^{(n)}\bar g)\|_\text{BV}$ and
$\|\Phi(\Lambda_\omega^{(n)}\bar h)\|_\text{BV}$ are uniformly bounded above in $n$. 
Since $\LL_\omega^{(n)}f=\Phi(\Lambda_\omega^{(n)}\bar g)-\Phi(\Lambda_\omega^{(n)}\bar h)$,
it follows that $\|\LL_\omega^{(n)}f\|_\text{BV}$ is uniformly bounded above in $n$.

On the other hand, if $\int f\ne 0$, then $\|\LL_\omega^{(n)}f\|_\text{BV}\ge
\|\LL_\omega^{(n)}f\|_1\ge |\int \LL_\omega^{(n)}f|=|\int f|\ne 0$, so that there is a
uniform lower bound also. It follows that 
$\lim_{n\to\infty} \tfrac1n \log \|\LL_\omega^{(n)}f\|_\text{BV}=0$. That is,
the top Lyapunov exponent is $\lambda^\ep_1=0$.

\subsubsection{Estimation of the second Lyapunov exponent}\label{s:2ndExp}
Let $\bar f=(f_0, \dots, f_k)\in X$.
We now study the evolution of the quantities $\phi^\pm(\bar f)=\int_{J^\pm}
(f_0+\ldots+f_k)$.
We first claim that if $\bar f\in\mathcal C$ and $\epsilon$ satisfies the constraints
above, then $\|f_0\|_1=O\big(\max(|\phi^+(\bar f)|,|\phi^-(\bar f)|\big)$.
To see this, notice that $|f_0|$ is at least $\tfrac 12\|f_0\|_1$ at some point $x_0\in J^+\cup J^-$. 
Without loss of generality, we assume $x_0\in J^+$ and $f(x_0)>0$. 
Now \ref{it:varf0} implies $|f_0(x)-f_0(x_0)|\le 33\epsilon
\|f_0\|_1$ for all $x\in J^+$, so that $f_0(x)\ge (\tfrac 12-33\epsilon)\|f_0\|_1$ on $J^+$. 
We now have $\phi^+(\bar f)=\int_{J^+}(f_0+\ldots+f_k)\ge (\tfrac 12-33\epsilon-4k\epsilon)\|f_0\|_1$
where we used \ref{it:L1fj} to estimate $\int_{J^+}f_1,\ldots,\int_{J^+}f_k$. In particular, for
all sufficiently small $\epsilon$ and all $\bar f\in\CC$
\begin{equation}\label{eq:phivsL1}
\tfrac 13\|f_0\|_1\le \max(|\phi^+(\bar f)|,|\phi^-(\bar f)|)\le \|f_0\|_1.
\end{equation}

Writing $\bar g=\Lambda_\omega(\bar f)$, we have, from the Perron-Frobenius property, that
$\int_{J^\pm}g_0=\int_{J^\pm\setminus H_\omega}f_0+\int_{J^\pm}f_k$; 
$\int_{J^\pm}g_1=\int_{H_\omega^\mp}f_0$ (that is, $\int_{J^+}g_1=\int_{H_\omega^-}f_0$ and
vice versa) and
$\int_{J^\pm}g_j=\int_{J^\pm}f_{j-1}$ for $j=2,\ldots,k$.

Summing, we obtain
$$
\phi^\pm(\Lambda_\omega\bar f)=\phi^\pm(\bar f)-\int_{H_\omega^\pm}f_0+\int_{H_\omega^\mp}f_0.
$$
(This equality is intuitively clear when one thinks of $\phi^\pm(\bar f)$ as the total mass on the left or right
side, respectively.)
Adding the two quantities, we obtain $\phi^+(\Lambda_\omega \bar f)+\phi^-(\Lambda_\omega \bar f)
=\phi^+(\bar f)+\phi^-(\bar f)$, that is the conservation of mass.
 Let $\mathcal C_0=\{\bar f\in \mathcal C\colon \phi^+(\bar f)+\phi^-(\bar f)=0\}$.
Thus, if $\bar f\in \CC_0$ then $\Lambda_\omega\bar f\in\CC_0$.

By \eqref{eq:f0J+} and \eqref{eq:f0J-}, 
$$
\left|\int_{H_\omega^\pm}f_0-m(H_\omega^\pm)\int_{J^\pm}f_0\right|
\le 33\epsilon m(H_\omega^\pm)\|f_0\|_1=O(\epsilon^2)\|f_0\|_1.
$$
By \ref{it:L1fj},
$$
\left|\int_{J^\pm}f_0-\phi^\pm(\bar f)\right|\le \frac{3k\epsilon}{(1-39\epsilon)^k}\|f_0\|_1=
O(\epsilon\logeps)\|f_0\|_1.
$$
Since $m(H_\omega^+)=\epsilon a(\omega)/(1+\epsilon a(\omega))=\epsilon a(\omega)+O(\epsilon^2)$
and $m(H_\omega^-)=\epsilon b(\omega)+O(\epsilon^2)$, we obtain
\begin{align*}
\phi^+(\Lambda_\omega\bar f)&=\phi^+(\bar f)-
\epsilon\big(a(\omega)\phi^+(\bar f)-b(\omega)\phi^-(\bar f)\big)
+O(\epsilon^2\logeps)\|f_0\|_1
\\
\phi^-(\Lambda_\omega\bar f)&=\phi^-(\bar f)+
\epsilon\big(a(\omega)\phi^+(\bar f)-b(\omega)\phi^-(\bar f)\big)
+O(\epsilon^2\logeps)\|f_0\|_1
\end{align*}

In the special case $\bar f\in \CC_0$. 
\begin{align*}
\phi^+(\Lambda_\om(\bar f))&=\Big(1-\epsilon\big(a(\omega)+b(\omega)\big)\Big)
\phi^+(\bar f)+O(\epsilon^2\logeps)\|f_0\|_1\\
&=\Big(1-\epsilon\big(a(\omega)+b(\omega)\big)+O(\epsilon^2\logeps)\Big)
\phi^+(\bar f).
\end{align*}

It follows that if $\bar f\in\CC_0$ then
\begin{equation}\label{eq:lambda2}
|\phi^+(\Lambda^{(n)}_\omega \bar f)|=
\prod_{j=0}^{n-1}\Big( 1-\epsilon\big(a(\sigma^j\omega)+b(\sigma^j\omega)\big)
+O(\epsilon^2\logeps)\Big)\phi^+(\bar f).
\end{equation}

Suppose $f\in\text{BV}$ and $\int f=0$. Then, choosing a sufficiently large $c$, 
$f$ may be expressed as $f=g+h$ with $g(x)=f(x)-c\,\sign(x)$, 
$h(x)=c\,\sign(x)$, and $\var_{0^c}g\le 33\epsilon\|g\|_1$. Hence 
$\bar f=(f,0,\ldots,0)$ can be written as $\bar g+\bar h$ where $\bar g=(g,0,\ldots,0)$
and $\bar h=(h,0,\ldots,0)$ with $\bar g,\bar h\in \CC_0$. 
Also, it is straightforward to check that $\LL_\omega^{(n)}f=\Phi\big(\Lambda_\omega^{(n)}(f,0,\ldots,0)\big)$. Hence we have
$\LL_\omega^{(n)}f=\Phi(\Lambda_\omega^{(n)}\bar g)+\Phi(\Lambda_\omega^{(n)}\bar h)$. 
Since $\bar g,\bar h\in \CC_0$, $\phi^+(\Lambda_\omega^{(n)}\bar g)$
and $\phi^+(\Lambda_\omega^{(n)}\bar h)$ decay as in \eqref{eq:lambda2}.
Since $\Lambda_\omega^{(n)}\bar g$ and $\Lambda_\omega^{(n)}\bar h$ belong to $\CC_0$,
\eqref{eq:phivsL1} implies that
$\|(\Lambda^{(n)}_\omega \bar g)_0\|_1$ and $\|(\Lambda^{(n)}_\omega\bar h)_0\|_1$ also have this
rate of decay, so by the definition of $\mathcal C$, $\|\Phi(\Lambda^{(n)}_\omega\bar g)\|_\text{BV}$
and $\|\Phi(\Lambda^{(n)}_\omega\bar h)\|_\text{BV}$ also decay at this rate, which ensures
that 
\begin{equation}\label{eq:CC0growth}
\begin{split}
\limsup_{n\to\infty} & \tfrac1n \log 
\|\LL_\omega^{(n)}f\|_\text{BV}\\ 
&\le  \limsup_{n\to\infty} \tfrac1n \log 
\prod_{j=0}^{n-1}\Big(1-\epsilon\big(a(\sigma^j\omega)+b(\sigma^j\omega)\big)+
O(\epsilon^2\logeps)\Big)\\
&=-\epsilon\int(a(\omega)+b(\omega))\,d\PP(\omega)+O(\epsilon^2\logeps),
\end{split}
\end{equation}
where the Birkhoff ergodic theorem has been used in the last line, together with Taylor's theorem.

We can obtain a corresponding lower bound:
Let $\bar f=(\sign,0,\ldots,0)$ so that $\bar f\in \bar C_0$. 
Then by \eqref{eq:lambda2}, $\phi_+(\Lambda_\omega^{(n)}\bar f)$ may be expressed as an orbit
product of $(1-\epsilon\big(a(\sigma^j\omega)+b(\sigma^j\omega)\big)+
O(\epsilon^2\logeps)$ terms, so that by the same arguments as above,
the growth rate for $\bar f$ has a lower bound that matches the upper bound of \eqref{eq:CC0growth}.

Since any two-dimensional subspace of BV contains a non-zero function with integral 0 (and therefore
with growth rate bounded as in \eqref{eq:CC0growth}), the multiplicity of $\lambda_1^\ep$ is 1 
and 
$$
\lambda_2^\ep= -\epsilon\int (a(\omega)+b(\omega))\,d\PP(\omega)+O(\epsilon^2\logeps),
$$
as required.

\step{Subsequent Lyapunov exponents}\label{s:3rdLexp}
We now show that all other Lyapunov exponents are at most $-\log 2+O(\epsilon\logeps)$. 

Let $\epsilon$ be small and fixed, as before, and let $2^k\epsilon<\tfrac 14\le 2^{k+1}\epsilon$, 
so that $2^{-k}\le 8\epsilon$. 
We record the following statement, which is a recapitulation of  \eqref{eq:lambda2}.

\begin{cor}\label{cor:scaling}
Let $f\in\mathrm{BV}$ satisfy $\int f=0$ and $\var_{0^c}f\le 33\epsilon\|f\|_1$.
Then, 
\begin{equation*}
\int_{J^+}\LL_\omega^{(n)}f=\prod_{j=0}^{n-1}
\left(1-\epsilon\big(a(\sigma^j\omega)+b(\sigma^j\omega)\big)+O(\epsilon^2\logeps)\right)\int_{J^+}f.
\end{equation*}
\end{cor}

\begin{proof}
The proof is by translating into and out of the space $X$. 
Let $f$ be as in the statement and let $\bar f=(f,0,\ldots,0)$. 
Then $\bar f\in \CC_0$, and the claimed equality is a direct restatement of \eqref{eq:lambda2}.
\end{proof}

We recall that $\sign=\mathbf{1}_{J^+}-\mathbf{1}_{J^-}$ denotes the sign function.
\begin{cor}\label{cor:signBVphi}
For any $\omega\in\Omega$ and $n\ge 0$, 
$$
\|\LL_\omega^{(n)}\sign\|_\text{BV}\le 15\int_{J^+}\LL_\omega^{(n)}\sign.
$$
\end{cor}

\begin{proof}
Let $\bar f=(\sign,0,\ldots,0)\in\CC_0$ and consider $\bar g=\Lambda^{(n)}_\omega\bar f$. 
As already noted, $\phi^+(\bar g)=\int_{J^+}\LL_\omega^{(n)}\sign$. By \eqref{eq:phivsL1},
$\|g_0\|_1\le 3\phi^+(\bar g)$. Using \ref{it:varfj}, \ref{it:L1fj} and \ref{it:varf0}, we see that
for small $\epsilon$, $\|\LL_\omega^{(n)}\sign\|_\text{BV}=
\|g_0+\ldots+g_k\|_\text{BV}\le 5\|g_0\|_1\le 15\phi^+(\bar g)=15\int_{J^+}\LL_\omega^{(n)}\sign$.
\end{proof}

\begin{lem}\label{lem:lambda3}
Let $ f\in \text{BV}$ satisfy $\int_{J^+}f=\int_{J^-}f=0$.
Then $\|\LL_\omega^k f\|_\text{BV}\le 144\epsilon|\mathrm{log\ }\epsilon|\| f\|_\text{BV}$.
\end{lem}

\begin{proof}
Let $f\in \text{BV}$ satisfy $\|f\|_\text{BV}=1$, $\int_{J^+}f=\int_{J^-}f=0$. 
The definition of the norm implies
\begin{align*}
	\|f\|_1&\le 1;\\
	\var_{0^c} f&\le 1.
\end{align*}
We push $f$ forward under $\LL_\omega^{(k)}$ in stages, keeping track of the parts that have leaked. 
Let $h^{(0)}=f$ and for $j=1,\ldots k$, define
\begin{align*}
g_j&=\LL_{\sigma^{j-1}\omega}(\mathbf 1_{H_{\sigma^{j-1}\omega}}h^{(j-1)})\text{;\quad and}\\
h^{(j)}&=\LL_{\sigma^{j-1}\omega}(\mathbf 1_{H_{\sigma^{j-1}\omega}^c}h^{(j-1)}),
\end{align*}
so that $\LL_{\sigma^{j-1}\omega}h^{(j-1)}=h^{(j)}+g_j$.
Combining these equalities inductively, we see
\begin{equation}\label{eq:decomp}
\LL_\omega^{(k)}f=\LL_\omega^{(k)}h^{(0)}=
h^{(k)}+\sum_{j=1}^{k}\LL_{\sigma^j\omega}^{(k-j)}g_j.
\end{equation}
This should be interpreted as the parts of $f$ that did not leak through the hole
during the first $k$ steps, together with
the parts of $f$ that leaked through on the $j$th step, for $j$ running from 1 to $k$.

By hypothesis, we have $\var_{0^c}h^{(0)}\le 1$. 
We see from the recursive definition (using the facts that the branches of $J^\pm\setminus
H_{\sigma^{j-1}\omega}$ map fully over $J^\pm$ under $T_{\sigma^{j-1}\omega}$ and
that $|T_{\sigma^{j-1}\omega}'|\ge 2$) that
$\var_{0^c}h^{(j)}\le \tfrac 12\var_{0^c}h^{(j-1)}$ 
and $|\int_{J^\pm}h^{(j)}|\le |\int_{J^\pm}h^{(j-1)}|+|\int_{J^\mp}g_j|$.
Hence $\var_{0^c}h^{(j)}\le 2^{-j}$ for $j=0,\ldots,k$ and
\begin{align*}
\left|\int_{J^+}h^{(j)}\right|&\le
\left|\int_{J^+}h^{(j-1)}\right|+m(J^+\cap H_{\sigma^{j-1}\omega})\|h^{(j-1)}|_{J^+}\|_\infty\\
&\le \left|\int_{J^+}h^{(j-1)}\right|+\epsilon\left(\left|\int_{J^+}h^{(j-1)}\right|+\var_{0^c}h^{(j-1)}\right)\\
&=(1+\epsilon)\left|\int_{J^+}h^{(j-1)}\right|+\epsilon 2^{-(j-1)},
\end{align*}
for $j=1,\ldots,k$; with an exactly similar inequality for $|\int_{J^-}h^{(j)}|$.
Since $\int_{J^\pm}h^{(0)}=0$, we deduce inductively
$$
\left|\int_{J^\pm}h^{(j)}\right|\le \epsilon\left((1+\epsilon)^{j-1}+(1+\epsilon)^{j-2}2^{-1}+\ldots
+2^{-(j-1)}\right).
$$
In particular, for sufficiently small $\epsilon$, since $k\approx
\log_2(\tfrac 1\epsilon)$, $|\int_{J^\pm}h^{(j)}|\le 3\epsilon$ for $j=1,\ldots,k$.
Combining this with the estimate for $\var_{0^c}h^{(j)}$, we see
$\|h^{(j)}\|_\infty\le 3\epsilon+2^{-j}\le 2\cdot 2^{-j}$ for each $j=1,\ldots,k$.

We now use this to estimate the $L^1$ norm and variation of $\LL_\omega^{(k)}f$.
For $j=1,\ldots,k$, we have $\|g_j\|_1\le 2\epsilon \|h^{(j-1)}\|_\infty\le 8\epsilon\cdot 2^{-j}$
and $\var_{0^c}g_j\le \tfrac 12\var_{0^c}h^{(j-1)}+2\|h^{(j-1)}\|_\infty\le 9\cdot 2^{-j}$, so 
that $\|\LL_{\sigma^j\omega}^{(k-j)}g_j\|_1\le 8\epsilon\cdot 2^{-j}$
and $\var_{0^c}\LL_{\sigma^j\omega}^{(k-j)}g_j\le 9\cdot 2^{-k}$.
Using \eqref{eq:decomp},
we see
\begin{align*}
\|\LL^{(k)}_\omega f\|_1&\le 2\cdot 2^{-k}+8\epsilon;\\
\var_{0^c}\LL^{(k)}_\omega f&\le 2^{-k}+9k\cdot 2^{-k}.
\end{align*}
Combining this gives $\|\LL^{(k)}_\om f\|_\text{BV}\le 144\epsilon\logeps$ as claimed.
\end{proof}

\begin{lem}\label{lem:invFunctional}
Let $\omega\in\Omega$. Then the sequence of functionals $(\psi^{(n)}_\omega)$ in $\mathrm{BV}^*$ given by
\begin{equation}\label{eq:invFunctional}
\psi^{(n)}_\omega(f)=\frac{\int_{J^+}(\LL_\omega^{(n)}(f-\tfrac 12\int f))}{\int_{J^+}(\LL_\omega^{(n)}\sign)}
\end{equation}
converges to a functional $\psi^*_\omega$.

If $f\in\mathrm{BV}$ satisfies $\int f=0$ and $\psi_\omega^*(f)=0$, then
$$
\limsup_{n\to\infty}\tfrac 1n\log\|\LL_\omega^{(n)}f\|_{\text{BV}}\le -\log 2+O(\log\log\tfrac1\epsilon
/\log\tfrac 1\epsilon).
$$
\end{lem}

\begin{proof}
We first note by Corollary \ref{cor:scaling} that $\int_{J^+}\LL_\omega^{(n)}\sign\ne 0$ for each $n$, so that
$\psi^{(n)}_\omega(f)$ is defined. Observe that $\psi_\omega^{(n)}(f+a\sign)=\psi_\omega^{(n)}(f)+a$ for any $a$, $n$ and $\omega$. 

Let $\omega\in\Omega$ and fix a function $f\in \text{BV}$ with the property that $\int f=0$ and
$\|f\|_\text{BV}=1$.
We let $a_n=\psi_\omega^{(n)}(f)$ for each $n$. 
Define $g_n=\LL_\omega^{(n)}(f-a_{n}\sign)$. By the observation above, $\psi_\omega^{(n)}(f-a_{n}\sign)=0$, so that $\int_{J^+}g_n=0$.
Since $\int g_n=0$, it follows that $\int_{J^-}g_n=0$ as well.
Hence, Lemma \ref{lem:lambda3} applies, giving
\begin{equation*}
\|\LL_{\sigma^{n}\omega}^{(k)}g_n\|_\text{BV}\le 144\epsilon|\text{log\,}\epsilon|\|g_n\|_\text{BV}.
\end{equation*}

Notice that 
$g_{n+m}=\LL_{\sigma^{n}\omega}^{(m)}(g_n-(a_{n+m}-a_n)\LL_\omega^{(n)}\sign)$, so that
\begin{equation}\label{eq:adiff}
a_{n+m}-a_n=\left(\int_{J^+}\LL_{\sigma^{n}\omega}^{(m)}g_n\right)
\Big/\left(\int_{J^+}\LL_\omega^{(n+m)}\sign\right).
\end{equation}
Also, 
\begin{equation}\label{eq:notsure}
\begin{split}
&\left|\int_{J^+}(a_{n+k}-a_n)\LL_\omega^{(n+k)}\sign\right|=\left|
\int_{J^+}\LL_{\sigma^{n}\omega}^{(k)}g_n\right|\\
&\le\|\LL^{(k)}_{\sigma^{n}\omega}g_n\|_\text{BV}
\le 144\epsilon\logeps\|g_n\|_\text{BV}.
\end{split}
\end{equation}

Hence by Corollary \ref{cor:signBVphi}, $|a_{n+k}-a_n|\|\LL_\omega^{(n+k)}\sign\|_\text{BV}\le
15\cdot 144\epsilon|\text{log\,}\epsilon|\|g_n\|_\text{BV}$, so that
\begin{equation}\label{eq:3decay}
\begin{split}
\|g_{n+k}\|_\text{BV}&\le \|\LL^{(k)}_{\sigma^{n}\omega}g_n\|_\text{BV}+\|(a_{n+k}-a_n)
\LL_\omega^{(n+k)}\sign\|_\text{BV}\\
&\le \ce\epsilon|\text{log\,}\epsilon|\|g_n\|_\text{BV}.
\end{split}
\end{equation}
It follows that $\|g_{nk}\|_\text{BV}\le (\ce\epsilon|\text{log\,}\epsilon|)^n$, while \S\ref{s:2ndExp} implies that
\begin{equation}\label{eq:2decay}
\|\LL_\omega^{(kn)}\sign\|_\text{BV}\ge \big(1-2\epsilon+O(\epsilon^2\logeps)\big)^{kn}.
\end{equation}
Using \eqref{eq:adiff}, this implies that 
\begin{align*}
a_{(n+1)k}-a_{nk}&=\left(\int_{J^+}\LL_{\sigma^{nk}\omega}^{(k)}g_{nk}\right) \Big/
\left(\int_{J^+}\LL_\omega^{(n+1)k}\sign\right)\\
&\le \frac{(\ce\epsilon\logeps)^{n}}{(1-3\epsilon)^{(n+1)k}}.
\end{align*}
Summing the geometric series, $|a_{nk}-\lim_{n\to\infty} a_{nk}|\le (\cg\epsilon\logeps)^n$.
From \eqref{eq:adiff}, we see that for each $0\le j<k$, 
$a_{nk+j}-a_{nk}$ also decreases exponentially in $n$, so that the full sequence $(a_n)$ converges to 
the same limit as $(a_{nk})$ and hence $\psi_\omega^{n}(f)$ is convergent as required.
Since the estimates in \eqref{eq:3decay} and \eqref{eq:2decay} are uniform in $\omega$, it follows that
$\psi_\omega^*(f)\le c\|f\|_\text{BV}$ for all $f\in\text{BV}$. 

Finally, let $f\in\text{BV}$ satisfy $\int f=0$ and $\psi^*_\omega(f)=0$. Let $(g_n)$ be as above, 
so that $g_n=\LL_\omega^{(n)}(f-\psi_\omega^{(n)}(f)\sign)$. We established that 
$\|g_n\|_{\BV}=O((\ce\epsilon\logeps)^{n/k})$ and $a_n=a_n-\psi_\omega^{(n)}(f)=
O\big((\cf\epsilon\logeps)^{n/k}\big)$, so that
\begin{align*}
\|\LL_\omega^{(n)}f\|_\text{BV}&\le\|g_n\|_\text{BV}+|a_n|\|\LL_\omega^{n}\sign\|_{\BV}\\
&=O\big((\cf\epsilon\logeps)^{n/k}\big).
\end{align*}
Taking $n$th roots, we see
$$
\limsup_{n\to\infty}\|\LL_\omega^{(n)}f\|_\text{BV}^{1/n}\le (\cf\epsilon\logeps)^{1/k}.
$$
Taking logarithms, the growth rate is at most
 $-\log 2+O(\log\log\tfrac 1\epsilon/\log\tfrac1\epsilon)$.
\end{proof}

An immediate consequence of Lemma~\ref{lem:invFunctional} is the following.
\begin{cor}\label{cor:3rdLexp}
If $V\subset \BV$ is any three-dimensional space, there exists $f\in V\setminus \{0\}$ satisfying  $\int f = 0$ and $\psi_\om^*(f) = 0$, so that 
$$\limsup_{n\to\infty}\tfrac1n \log \|\LL_\omega^{(n)}f\|_\text{BV}\leq -\log 2+O(\log\log\tfrac 1\epsilon/\log\tfrac1\epsilon).$$
\end{cor}

In summary, 0 and $\lam_2^\ep$ are the only two Lyapunov exponents greater than $-\log 2+O(\log\log\tfrac 1\epsilon/\log\tfrac1\epsilon)$, counted with multiplicity. This concludes the proof of Theorem~\ref{mainthm}.

\section{A multiplicative ergodic theory (MET) framework for Theorem~\ref{mainthm}}\label{s:METs}
Besides BV, there are other function spaces which are well suited for the investigation of transfer operators of piecewise smooth interval maps, including paired tent maps. An example examined in  \cite{Thomine} is given by fractional Sobolev spaces $\hpt$,
and there are versions of the MET, \cite{GTQuas1,GTQuasJMD}, which apply to transfer operators on such spaces. Furthermore \cite[Lemma 3.16]{GTQuas1} implies that for any $0<\delta<\log 2$
one can select $p$ and $t$ near one ($t<\tfrac1p<1$), so that the MET applies and yields an $\hpt$ Oseledets splitting in the setting of Theorem~\ref{mainthm}, with index of compactness $\ka^\ep$ bounded above by  $-\log 2+ \delta$.

The following result gives a way to interpret Theorem~\ref{mainthm} in a multiplicative ergodic theory context when $a$ and/or $b$ have uncountable range. In particular, it implies measurability of the first and second Oseledets spaces, for instance when regarded as functions from $\Om$ to $L^1$.
\begin{thm}\label{thm:hpt}
Let $0<\delta<\log 2$ and assume the hypotheses of Theorem~\ref{mainthm} hold.
Then, for $\ep>0$ sufficiently small, $0$ and $\lambda_2^\ep$
are the only 
exceptional Lyapunov exponents of the $\hpt$ cocycle ${\mathcal L_\omega^\epsilon}^{(n)}$ which are greater than $-\log 2 + \delta$. Furthermore, they have multiplicity one.  
  \end{thm}

\begin{proof}
Let $\mu_1^\ep\geq \mu_2^\ep \geq \dots \geq \mu_j^\ep$ be the 
exceptional Lyapunov exponents of the $\hpt$ cocycle ${\mathcal L^\epsilon}^{(n)}_\omega$ greater than $-\log 2 + \delta$.
The fact that $\mu_1^\ep=0$ follows as in  \S\ref{s:1stExp}.
Since 
smooth functions are dense in both $\hpt$ and BV, the \hpt-MET implies that one is able to observe the largest $\hpt$ growth rate by considering BV functions with integral zero.
Also, since $\hpt\subset L^1$ is continuously embedded, it follows from \cite[Theorem 3.3]{FroylandStancevic13} that the growth rate of functions in the
Oseledets spaces associated to exceptional Lyapunov exponents measured with respect to $\hpt$ and $L^1$ norms coincide. Therefore, since  $\BV$ is stronger than $L^1$, $\mu_2^\ep\leq \lambda_2^\ep$.

In \S\ref{s:2ndExp}, we have shown that ${\mathcal L^\epsilon}^{(n)}_\omega\sign$  has growth rate $\lambda_2^\ep$ in both $\BV$ and $L^1$ norms. Since $\sign\in \hpt$ and $\hpt$ is stronger than $L^1$, this implies $\mu_2^\ep \geq \lambda_2^\ep$. Combining with the previous paragraph, $\mu_2^\ep=\lambda_2^\ep$.
     
Recall that Corollary~\ref{cor:3rdLexp} ensures that in any three dimensional subspace of $\BV$ functions there exists a non-zero function with decay rate bounded above by $-\log 2+O(\log\log\tfrac 1\epsilon/\log\tfrac 1\epsilon)$ in BV, and hence also in $L^1$.   
Using  once again the density argument, and the fact that Lyapunov exponents measured with respect to $\hpt$ and $L^1$ norms coincide on Oseledets spaces,
 we conclude that if $\ep$ is sufficiently small, then, in \hpt, there are no further exceptional Lyapunov exponents greater than $-\log 2 +\delta$, apart from 0 and $\lambda_2^\ep$. Furthermore, these have multiplicity one.

\end{proof}

\section*{Acknowledgments}

We would like to acknowledge support from the Australian Research Council  (CGT)
and NSERC (AQ).

\bibliographystyle{abbrv}

\def\cprime{$'$}

\end{document}